\documentclass[12pt,reqno]{amsart}

\usepackage{amsmath, amssymb, amsthm}
\usepackage{graphicx, color}
\usepackage{enumerate}

\begin{document}

 \bibliographystyle{plain}
 \newtheorem{theorem}{Theorem}
 \newtheorem{lemma}[theorem]{Lemma}
 \newtheorem{corollary}[theorem]{Corollary}
 \newtheorem{problem}[theorem]{Problem}
 \newtheorem{conjecture}[theorem]{Conjecture}
 \newtheorem{definition}[theorem]{Definition}
 \newtheorem{prop}[theorem]{Proposition}
 \numberwithin{equation}{section}
 \numberwithin{theorem}{section}

 \newcommand{\mo}{~\mathrm{mod}~}
 \newcommand{\mc}{\mathcal}
 \newcommand{\rar}{\rightarrow}
 \newcommand{\Rar}{\Rightarrow}
 \newcommand{\lar}{\leftarrow}
 \newcommand{\lrar}{\leftrightarrow}
 \newcommand{\Lrar}{\Leftrightarrow}
 \newcommand{\zpz}{\mathbb{Z}/p\mathbb{Z}}
 \newcommand{\mbb}{\mathbb}
 \newcommand{\B}{\mc{B}}
 \newcommand{\cc}{\mc{C}}
 \newcommand{\D}{\mc{D}}
 \newcommand{\E}{\mc{E}}
 \newcommand{\F}{\mathbb{F}}
 \newcommand{\G}{\mc{G}}
  \newcommand{\ZG}{\Z (G)}
 \newcommand{\FN}{\F_n}
 \newcommand{\I}{\mc{I}}
 \newcommand{\J}{\mc{J}}
 \newcommand{\M}{\mc{M}}
 \newcommand{\nn}{\mc{N}}
 \newcommand{\qq}{\mc{Q}}
 \newcommand{\PP}{\mc{P}}
 \newcommand{\U}{\mc{U}}
 \newcommand{\X}{\mc{X}}
 \newcommand{\Y}{\mc{Y}}
 \newcommand{\itQ}{\mc{Q}}
 \newcommand{\sgn}{\mathrm{sgn}}
 \newcommand{\C}{\mathbb{C}}
 \newcommand{\R}{\mathbb{R}}
 \newcommand{\T}{\mathbb{T}}
 \newcommand{\N}{\mathbb{N}}
 \newcommand{\Q}{\mathbb{Q}}
 \newcommand{\Z}{\mathbb{Z}}
 \newcommand{\A}{\mathcal{A}}
 \newcommand{\ff}{\mathfrak F}
 \newcommand{\fb}{f_{\beta}}
 \newcommand{\fg}{f_{\gamma}}
 \newcommand{\gb}{g_{\beta}}
 \newcommand{\vphi}{\varphi}
 \newcommand{\whXq}{\widehat{X}_q(0)}
 \newcommand{\Xnn}{g_{n,N}}
 \newcommand{\lf}{\left\lfloor}
 \newcommand{\rf}{\right\rfloor}
 \newcommand{\lQx}{L_Q(x)}
 \newcommand{\lQQ}{\frac{\lQx}{Q}}
 \newcommand{\rQx}{R_Q(x)}
 \newcommand{\rQQ}{\frac{\rQx}{Q}}
 \newcommand{\elQ}{\ell_Q(\alpha )}
 \newcommand{\oa}{\overline{a}}
 \newcommand{\oI}{\overline{I}}
 \newcommand{\dx}{\text{\rm d}x}
 \newcommand{\dy}{\text{\rm d}y}
\newcommand{\diam}{\operatorname{diam}}
\newcommand{\bx}{\mathbf{x}}
\newcommand{\Ps}{\varphi}

\parskip=0.5ex

\title[Bounded remainder sets and cut-and-project sets II]{Constructing bounded remainder sets and cut-and-project sets which are bounded distance to lattices, II}
\author{Alan~Haynes,
Michael~Kelly,
Henna~Koivusalo}
\thanks{AH, HK: Research supported by EPSRC grants L001462, J00149X, M023540.\\
\phantom{K~}MK: Research supported NSF grants DMS-1101326, 1045119 and 0943832.\\
\phantom{K~}HK: Research supported by Osk. Huttunen Foundation.}

\allowdisplaybreaks

\maketitle

\begin{abstract}
Recent results of several authors have led to constructions of parallelotopes which are bounded remainder sets for totally irrational toral rotations. In this brief note we explain, in retrospect, how some of these results can easily be obtained from a geometric argument which was previously employed by Duneau and Oguey in the study of deformation properties of mathematical models for quasicrystals.
\end{abstract}

\section{Introduction}
Let $T:\T^s\rar\T^s$ be a Lebesgue measure preserving transformation of the $s$-dimensional torus $\T^s=\R^s/\Z^s$. A measurable set $A\subseteq\R^s$ is a {\it bounded remainder set} (henceforth denoted BRS) for $T$ if
\[
\sup_{x\in\R^s}\sup_{N\in\N}\left|\sum_{n=0}^{N-1}\chi_A(T^n(x)) - N|A|\right|<\infty,
\]
where $\chi_A:\T\rar\Z_{\ge 0}$ is the indicator function of $A$, viewed as a multi-set in $\T^s$ (i.e. so that $\chi_A$ could potentially take any non-negative integer value), and $|A|$ denotes its measure. Throughout this paper we will assume that $T$ is a toral rotation, given by $T(x)=x+\alpha$ for some $\alpha\in\R^s$. This is a situation which is particularly important in Diophantine approximation, and over the course of the past century it has been studied by a number of authors.

For $s=1$ the problem of classifying BRS's is satisfactorily dealt with by works of Hecke \cite{Heck1922}, Ostrowski \cite{Ostr1927/30}, and Kesten \cite{Kest1966/67} (see also related results of Oren \cite{Oren1982}), which together show that for an irrational rotation of $\T$ by $\alpha$, a necessary and sufficient condition for an interval $\mc{I}$ to be a BRS is that $|\mc{I}|\in\alpha\Z+\Z.$ Several papers \cite{Liar1987,Szus1954,Zhur2005,Zhur2011,Zhur2012} have investigated the corresponding problems in higher dimensions. Of particular note are the works of Sz\"{u}sz \cite{Szus1954}, who demonstrated a construction of parallelogram BRS's when $s=2$, and Liardet \cite[Theorem 4]{Liar1987}, who used a dynamical cocycles argument to extend Sz\"{u}sz's construction to arbitrary $s>1$. Other significant connections between BRS's and dynamical systems have been highlighted in \cite{GrabHellLiar2012, HaynKoiv2016,KellSadu2015,Rauz1972}, and a more thorough exposition of the history of these sets can be found in the introduction of \cite{GrepLev2015}.

Our understanding of polytope BRS's in higher dimensions has recently been substantially improved by Grepstad and Lev in \cite{GrepLev2015}. One of their central results is the following theorem.
\begin{theorem}\label{thm.BRSs}
Suppose that $\alpha\in\R^s$ and that $1,\alpha_1,\ldots ,$ and $\alpha_s$ are linearly independent over $\Q$. Then for any choice of linearly independent vectors $v_1,\ldots ,v_s\in\alpha\Z+\Z^s$, the parallelotope
\begin{equation}\label{eqn.Parallelotope}
P=\left\{\sum_{j=1}^st_jv_j:0\le t_j<1\right\}
\end{equation}
is a BRS for the rotation of $\T^s$ by $\alpha$.
\end{theorem}
Grepstad and Lev's proof of Theorem \ref{thm.BRSs} involved a detailed Fourier analytic argument which allowed them to calculate the `transfer functions' of the parallelotopes $P$. One of the goals of this paper is to explain how the theorem can be deduced using a simple and elegant geometric argument, first introduced by Duneau and Oguey \cite{DuneOgue90} in the study of regularity properties of quasicrystalline materials.  When applied to the BRS problem, the essence of their argument shows that the collection of return times of a toral rotation to a region of the form \eqref{eqn.Parallelotope} has a group structure coming from a natural higher dimensional realization of the problem.

Our main theorem, which is an abstract but otherwise straightforward generalization of \cite[Theorem 3.1]{DuneOgue90}, is Theorem \ref{thm.AbstractMain} below. Upon first reading, the statement of the theorem may seem rather complicated. Our reason for presenting it in such a form is to highlight its flexibility in terms of potential applications. The theorem applies not only to orbits of toral rotations, which can be recast as one dimensional quasicrystals, but also to multidimensional point patterns in Euclidean space. As a particular example, the following result will be shown to be an easy consequence of our main result.
\begin{theorem}\label{thm.BD}
Let $Y\subseteq\R^2$ be the collection of vertices of any Penrose tiling. Then there is a bijection from $Y$ to a lattice, which moves every point by at most a uniformly finite amount.
\end{theorem}
This result was mentioned in the paper of Duneau and Oguey (see comments on \cite[p.13]{DuneOgue90}), and it has recently been proved using different means by Solomon \cite{Solo2011}. We wish to emphasize that Theorem \ref{thm.BD} is only representative of the potential applications of our main theorem. The same argument used in its proof applies to any generic canonical cut and project set in Euclidean space. However, to avoid a long list of definitions and technical points regarding these objects, we maintain focus on the Penrose tilings, viewed as linear transformations of 5 to 2 cut and project sets.

It is also worth noting that the hypotheses Theorem \ref{thm.AbstractMain} allow us the flexibility to formulate results in many interesting non-Euclidean spaces. This is a topic which we believe will lead to new results, and which we leave open for further exploration.\vspace*{.1in}

\noindent{\em Acknowledgements:} AH wishes to thank Dirk Frettl\"{o}h for helpful conversations in Delft regarding the context of the results in this paper, and for sharing a preprint (joint work between him and Alexey Garber) in which similar results have been obtained.

\section{Statement and proof of main theorem}
Let $X$ be a metrizable $\Q$-vector space with a translation invariant metric $\mathrm{d}:X\times X\rar [0,\infty),$ and suppose that $V_p$ and $V_i$ are complementary $\Q$-vector subspaces of $X$, in the sense that $V_p\cap V_i=\{0\}$ and
\[X=V_p+V_i,\]
where the right hand side denotes the Minkowski sum of the two sets. In all of what follows, subsets of $X$ will be taken with the usual induced topology. Let $\rho_p$ and $\rho_i$ be the projections according to the above decomposition, onto $V_p$ and $V_i$, respectively. Suppose that $\Gamma$ is a subgroup of $X$ which is also a finite dimensional $\Z$-module. Given a compact $W\subseteq V_i$, we define a set $\tilde{Y}\subseteq X$ and a multi-set $Y\subseteq V_p$ by
\[\tilde{Y}=\rho_i^{-1}(W)\cap\Gamma\quad\text{and}\quad Y=\rho_p(\tilde{Y}).\]
We emphasize that $Y$ is a multi-set, so that if the restriction of $\rho_p$ to $\tilde{Y}$ is not injective, the elements of $Y$ are listed with the appropriate multiplicity.
\begin{theorem}\label{thm.AbstractMain}
  With the notation above, suppose further that:
  \begin{enumerate}
    \item[(i)] $Z$ is a subspace of $X$ (possibly different than $V_i$) which is complementary to $V_p$, $\phi_p:X\rar V_p$ is the projection onto $V_p$ with respect to the decomposition $X=V_p+Z$,
    \item[(ii)] The group $\Gamma$ has a decomposition of the form
    \[\Gamma=\Lambda+\Lambda_c,\]
    where $\Lambda$ and $\Lambda_c$ are groups with the properties that $\Lambda\cap\Lambda_c=\{0\}$ and $Z\cap\Gamma=\Lambda$, and
    \item[(iii)] $W$ is the image under $\rho_i$ of a compact fundamental domain for $Z/\Lambda'$, where $\Lambda'\leqslant\Lambda$ and $[\Lambda:\Lambda']=N<\infty.$
  \end{enumerate}
Then there is a bijection $f$ from $Y$ to a subgroup $\Lambda_c'$ of $V_p$ satisfying $\phi_p(\Lambda_c)\leqslant\Lambda_c'$ and $[\Lambda_c':\phi_p(\Lambda_c)]=N,$ with the property that $f$ moves every point of $Y$ by at most a uniformly finite amount.
\end{theorem}
\begin{proof}
  From assumption (ii) it follows that $\tilde{Y}$ can be written as the disjoint union
  \[\tilde{Y}=\bigsqcup_{\lambda\in\Lambda_c}(Z+\lambda)\cap\tilde{Y}.\]
  Let $W_Z=\rho_i^{-1}(W)\cap Z$. For $\lambda\in\Lambda_c$ consider the set
  \[W_\lambda=(Z+\lambda)\cap\rho_i^{-1}(W)-\lambda.\]
  Each such set is a translate of $W_Z$ in $Z$, and therefore a fundamental domain for $Z/\Lambda'$. It follows that each set $W_\lambda$ contains exactly $N$ points of $\Lambda,$ and from this we have that
  \[|(Z+\lambda)\cap \tilde{Y}|=N.\]
  For each $\lambda\in\Lambda_c$ let us write
  \[(Z+\lambda)\cap \tilde{Y}=\{y_\lambda(0),\ldots ,y_\lambda(N-1)\}.\]
  Choose a $\Z$-basis $\gamma_1,\ldots ,\gamma_d$ for $\Lambda_c$, and for $1\le j\le d$ let $\gamma_j'=\phi_p (\gamma_j)$. From the assumptions in hypothesis (ii), together with the fact that $X$ is a $\Q$-vector space, it is easy to deduce that the map $\phi_p$ is injective. It follows that the elements $\frac{1}{N}\gamma_1',\gamma_2',\ldots ,\gamma_d'$ generate a group $\Lambda_c'\leqslant V_p$ which contains $\phi_p(\Lambda_c)$ as a subgroup of index $N$.

  Now, using the decomposition
  \[\tilde{Y}=\bigsqcup_{\lambda\in\Lambda_c}\bigsqcup_{j=0}^{N-1}y_\lambda (j),\]
  we define a bijection $g:\tilde{Y}\rar\Lambda_c'$ by
  \[g(y_\lambda(j))=\phi_p(\lambda)+\frac{j}{N}\gamma_1',\]
  and we use this to define $f:Y\rar\Lambda_c'$ by
  \[f(y)=g(\rho_p^{-1}(y)\cap \tilde{Y}).\]
  Recall that $Y$ is a multi-set, so the map $f$ does produce a well defined bijection between $Y$ and $\Lambda_c'$.

  All that remains is to show that $f$ moves every point by at most a uniformly finite amount. Let $y\in Y$ and write $\tilde{y}=\rho_p^{-1}(y)\cap\tilde{Y}$ in two different ways, as
  \[\tilde{y}=v_i+v_p^{(1)},\quad\text{where}~v_i\in V_i~\text{and}~v_p^{(1)}\in V_p,\]
  and
  \[\tilde{y}=z+v_p^{(2)},\quad\text{where}~z\in Z~\text{and}~v_p^{(2)}\in V_p,\]
  so that $y=v_p^{(1)}$ and $f(y)=v_p^{(2)}+(j/N)\gamma_1'$, for some $0\le j<N$. Notice that, because of the definition of $\tilde{Y}$, we have that
  \[v_i\in W\quad\text{and}\quad z\in W_Z.\]
  Therefore we have that
  \begin{align*}
  \mathrm{d}(f(y),y)&\le \mathrm{d}\left(\tilde{y}-z+\frac{j}{N}\gamma_1',\tilde{y}-z\right)+\mathrm{d}(\tilde{y}-v_i,\tilde{y}-z)\\
  &\le \mathrm{d}\left(\frac{j}{N}\gamma_1',0\right)+\sup_{v\in W}\sup_{z'\in W_Z}\mathrm{d}(-v,-z').
  \end{align*}
  Since $W\cup W_Z$ is a compact subset of $X,$ the right hand side of this inequality is finite. Since it does not depend on $y$, the conclusion follows.
\end{proof}

\section{Proof of Theorem \ref{thm.BRSs}}
 For the proof of Theorem \ref{thm.BRSs} we will apply Theorem \ref{thm.AbstractMain} with $X=\R^{s+1},~\Gamma=\Z^{s+1}$,
  \[V_p=\langle(\alpha,1)\rangle_\R,\quad\text{and}\quad V_i=\langle e_1,\ldots ,e_s\rangle_\R,\]
  where $e_j$ denotes the $j$-th standard basis vector for $\R^{s+1}$. Write each of the vectors $v_j$ from the statement of Theorem \ref{thm.BRSs} as
  \[v_j=n^{(j)}-n^{(j)}_{s+1}\alpha,\]
  with $n^{(j)}_{s+1}\in\Z$ and $n^{(j)}\in\Z^s$, and let $W=P-x$ where $x\in V_i$ and $P$ is the  parallelotope of \eqref{eqn.Parallelotope}, realized as a subset of $V_i$. If we set $\lambda_j=(n^{(j)},n^{(j)}_{s+1})\in\Gamma$ then for each $j$ we have that
  \[v_j=\rho_i(\lambda_j).\]
  We take $Z$ to be the real subspace of $X$ generated by the vectors $\lambda_1,\ldots ,\lambda_s$, and we set $\Lambda=Z\cap \Gamma$. It is a straightforward exercise to show that we can find a vector $\lambda_c\notin Z$ for which
  \[\Gamma=\Lambda+\Lambda_c,\quad\text{with}\quad \Lambda_c=\langle\lambda_c\rangle_\Z.\]
  Therefore the hypotheses of Theorem \ref{thm.AbstractMain} are satisfied, allowing us to conclude that the corresponding multi-set $Y$ is in bijection, via a bounded distance map, with a lattice in $V_p$.

Now let $Y'$ be the multi-set defined by
  \[Y'=\{n\in\Z:n\alpha\in P-x~\mathrm{mod}~\Z^s\}\subseteq\R,\]
where each integer $n\in Y'$ is counted with multiplicity equal to the number of integer vectors $m\in\Z^s$ for which $n\alpha+m\in P-x$. Applying a linear transformation to $X$ in our above argument, we deduce that $Y'$ is in bounded distance bijection with a set of the form $\gamma\Z\subseteq\R$, for some $\gamma>0$. Since $1,\alpha_1,\ldots ,$ and $\alpha_s$ are $\Q$-linearly independent, the Birkhoff Ergodic Theorem (alternatively, Weyl's criterion from uniform distribution theory) implies that $\gamma=|P|^{-1}$. Write
  \[Y'=\{n_j\}_{j\in\Z},\]
  with $n_0\le 0< n_1$ and $n_j\le n_{j+1}$ for all $j$. Given $N\in\N$, choose $K\ge 0$ so that $n_K\le N<n_{K+1}$. Then we have
\[\sum_{n=1}^N\chi_P(n\alpha+x)=K,\]
and also that
\[K=|P|n_K+O(1)=|P|N+O(1).\]
After rearranging terms and taking the supremum over all $x$, this shows that $P$ is a bounded remainder set.

\section{Proof of Theorem \ref{thm.BD}}
The argument we will use in our proof of Theorem \ref{thm.BD} closely follows that outlined in \cite[p.13]{DuneOgue90}, the only difference being that we do not appeal to \cite[Theorem 4.1]{DuneOgue90}, but instead use the following result due to Laczkovich (see \cite[Theorem 1.1]{Lacz1992}).
\begin{theorem}\label{thm.Lacz}
For any discrete set $S\subseteq \R^d$ and for any $\kappa>0$ the following conditions are equivalent:
\begin{enumerate}
  \item[(i)] There is a positive constant $C>0$ with the property that, for any set $\mc{C}\subseteq\R^d$ which is a finite union of unit cubes, we have that
      \[\left|\#(S\cap\mc{C})-\kappa|\mc{C}|_d\right|\le C\cdot |\partial \mc{C}|_{d-1},\]
      where $|\cdot|_k$ denotes $k$-dimensional Lebesgue measure.
  \item[(ii)] There is a bijection from $S$ to $\alpha^{-1/d}\Z^d$, which moves every point by at most a uniformly finite amount.
\end{enumerate}
\end{theorem}

Let $\zeta=\exp(2\pi i/5)$, take $X=\R^5$, and let $V_p$ be the two dimensional real subspace of $X$ generated by the vectors
\[(1,\mathrm{Re}(\zeta),\mathrm{Re}(\zeta^2),\mathrm{Re}(\zeta^3),\mathrm{Re}(\zeta^4))\]
and
\[(0,\mathrm{Im}(\zeta),\mathrm{Im}(\zeta^2),\mathrm{Im}(\zeta^3),\mathrm{Im}(\zeta^4)).\]
Take $V_i$ to be the real subspace of $X$ orthogonal to $V_p$, and let $W\subseteq V_i$ be a translate of the image under $\rho_i$ of the unit cube in $\R^5$. Well known results of de Bruijn \cite{Brui1981} and Robinson \cite{Robi1996} show that the set $Y$ obtained in this way is the image under a linear transformation of the collection of vertices of a Penrose tiling, and in fact that all Penrose tilings can be obtained in a similar way.

The set $W$ can be written as a disjoint union of $\binom{5}{3}=10$ parallelotopes $W_1,\ldots ,W_{10},$ each of which is a translate of the projection to $V_i$ of a parallelotope in $\R^5$ generated by $3$ elementary basis vectors. Each parallelotope $W_j$ satisfies the hypotheses of Theorem \ref{thm.AbstractMain}, and therefore gives rise to a set $Y_j$ which is in bijection, via a bounded distance map, with a lattice.

Appealing to Theorem \ref{thm.Lacz}, for each $j$ and for any region $\mc{C}\in V_p$ which is a finite union of unit cubes (chosen with respect to some basis for $V_p$), the number of points in $Y_j\cap\mc{C}$ is equal to the expected number plus an error which is at most a constant multiple of the volume of the boundary. Therefore, by summing over all $j$, we see that the set $Y$ also satisfies condition (i) of Theorem \ref{thm.Lacz}, and the conclusion of Theorem \ref{thm.BD} follows immediately.

\vspace{.1in}

{\footnotesize
\noindent AH, HK\,:\\
Department of Mathematics, University of York,\\
Heslington, York, YO10 5DD, England\\
e-mails: alan.haynes@york.ac.uk, henna.koivusalo@york.ac.uk

\vspace*{.1in}

\noindent MK\,:\\
Department of Mathematics, University of Michigan\\
Ann Arbor, MI, 48109, USA\\
e-mails: michaesk@umich.edu

}

\end{document}